\documentclass{birkjour}
\usepackage[utf8]{inputenc}

\usepackage{enumerate,physics}
\usepackage{mathscinet}
\usepackage[noadjust]{cite}

\newtheorem{theorem}{Theorem}%

\newtheorem{proposition}[theorem]{Proposition}
\newtheorem{corollary}[theorem]{Corollary}
\theoremstyle{definition}
\newtheorem{definition}[theorem]{Definition}
\theoremstyle{remark}
\newtheorem{remark}[theorem]{Remark}

\newcommand{\ad}{{\rm ad}}

\newcommand{\Ad}{{\rm Ad}}

\begin{document}

\title[Banach Poisson--Lie group structure on
$\operatorname{U}(\mathcal{H})$]{Banach Poisson--Lie group structure
  on $\operatorname{U}(\mathcal{H})$}

\author[A.B.~Tumpach]{Alice Barbora Tumpach}

\address{UMR CNRS 8524\\
  UFR de Math\'ematiques\\
  Laboratoire Paul Painlev\'e\\
  59 655 Villeneuve d'Ascq Cedex\\
  France\\ \& \\ Institut CNRS Pauli\\ UMI CNRS 2842\\
  Oskar-Morgenstern-Platz 1 \\1090 Wien\\Austria}

\email{alice-barbora.tumpach@univ-lille.fr}

\author[T.~Goli\'nski]{Tomasz Goli\'nski}

\address{University of Bia\l ystok\\ Cio\l kowskiego 1M\\15-245 Bia\l
  ystok\\ Poland}

\email{tomaszg@math.uwb.edu.pl}

\thanks{This research was partially supported by joint National
  Science Centre, Poland (number 2020/01/Y/ST1/00123) and Fonds zur
  Förderung der wissenschaftlichen Forschung, Austria (number I
  5015-N) grant ``Banach Poisson--Lie groups and integrable
  systems''.}

\begin{abstract}
  We construct a Banach Poisson--Lie group structure on the unitary
  group of a separable complex Hilbert space.
\end{abstract}



\maketitle

\section{Introduction}

\paragraph{Notation} In this paper we consider the Banach Lie group of
bounded unitary operators $\operatorname{U}(\mathcal{H})$ on a complex
separable Hilbert space $\mathcal{H}$. We denote by
$L_{\infty}(\mathcal{H})$ the Banach space of bounded linear operators
on $\mathcal{H}$, and by $L_1(\mathcal{H})$ the Banach algebra of
trace class operators on $\mathcal{H}$. The Banach Lie algebra of
$\operatorname{U}(\mathcal{H})$ consisting of skew-hermitian bounded
operators will be denoted by $\mathfrak{u}(\mathcal{H})$ and the
Banach Lie algebra of trace-class skew-hermitian operators by
$\mathfrak{u}_1(\mathcal{H})$.

In the whole paper for a Banach space $\mathfrak{b}$ we will use the
notation $\mathfrak{b}^*$ to denote the continuous dual of
$\mathfrak{b}$, i.e. the Banach space of continuous functionals on
$\mathfrak{b}$, and $\mathfrak{b}_*$ for a predual of $\mathfrak{b}$,
i.e. for a Banach space such that
$(\mathfrak{b}_*)^* \cong \mathfrak{b}$.

All manifolds in the paper are assumed to be of $C^\infty$ class and all considered objects (e.g. sections, functions) are smooth.

\paragraph{Aim of the paper} The aim of this paper is to define a
structure of Banach Poisson--Lie group on
$\operatorname{U}(\mathcal{H})$ defined on the pre-cotangent bundle
$T_*\operatorname{U}(\mathcal{H})$, with fibers modeled on the Banach
quotient space $L_1(\mathcal{H})/\mathfrak{u}_1(\mathcal{H})$. Notably
$L_1(\mathcal{H})/\mathfrak{u}_1(\mathcal{H})$ inherits a Lie algebra
structure from this construction.

\paragraph{Related work} The notion of Poisson manifold in the context
of Banach manifolds was introduced in \cite{OR} and generalized in
various directions in
\cite{pelletier,pelletier19,neeb14,debievre15,BGT,
  beltita-odzijewicz,tumpach-bruhat}. The notion of Poisson--Lie group
in the finite-dimensional setting goes back to
\cite{drinfeld83,semenov87,kosmann88,lu-weinstein90}. The notion of
Banach Poisson--Lie group was introduced in \cite{tumpach-bruhat} and
examples related to the Korteweg--de Vries hierarchy and restricted
Grassmannian \cite{Ratiu-grass,GO-grass} were investigated. Some
other, more formal approaches to infinite dimensional Poisson--Lie
groups can be found e.g. in
\cite{grabowski94,zakharevich94,khesin95}. The geometry of the unitary
groups was studied e.g. in \cite{grabowski05,andruchow10,beltita21}.

\section{Definition of Banach Poisson--Lie groups}

We recall in this section the generalization of the definition of
Banach Poisson manifolds adapted to our considerations. The definition
given below was introduced in \cite{tumpach-bruhat} and called
generalized Banach Poisson manifolds. In order to be coherent with the
terminology used in \cite{pelletier}, we will call this structure
Banach sub-Poisson manifold. This notion is a generalization of the
notion of Banach Poisson manifolds given in \cite{OR} to the case
where the Poisson tensor is only defined on a subset of the cotangent
bundle (Definition~\ref{Poisson_tensor}). This subset will be a bundle
with possibly different topology and large enough that it is in
duality with the tangent bundle (Definition~\ref{duality_bundle}). The
definition of Banach Poisson--Lie groups in this context is given in
Definition~\ref{BPLG}. In the finite-dimensional case, all these
definitions become the usual ones.

\begin{definition}\label{duality_bundle}
  Let $M$ be a Banach manifold. We will say that a Banach bundle
  $\mathbb{F}$ over $M$ is \textbf{in duality} with the tangent bundle
  to $M$ if, for every $p\in M$, there is a duality pairing
  (i.e. non-degenerate continuous bilinear map) between the fibers
  $\mathbb{F}_p$ and $T_pM$, which depends smoothly on $p$.
\end{definition}

\begin{remark}
  Any Banach bundle $\mathbb{F}$ over $M$ in duality with $TM$ injects
  continuously into $T^*M$, hence we will identify it sometimes with a
  subset of $T^*M$. In this way local sections of $\mathbb{F}$ will be regarded as local 1-forms on $M$. Such a bundle $\mathbb{F}$ will play the role of
  \textbf{co-characteristic distribution} in the sense of
  \cite{beltita-odzijewicz}. However in general it may not be a Banach
  subbundle of $T^*M$.
\end{remark}

We will denote by $\Lambda^2\mathbb{F}^{*}$ the vector bundle over $M$
whose fiber over $p\in M$ is the Banach space of continuous
skew-symmetric bilinear forms on the fiber $\mathbb{F}_p$. %

\begin{definition}\label{Poisson_tensor}
  Let $M$ be a Banach manifold and $\mathbb{F}$ a bundle in duality
  with $TM$. A smooth section $\pi$ of $\Lambda^2\mathbb{F}^*$ is
  called a \textbf{Poisson tensor} on $M$ with respect to $\mathbb{F}$
  if~:
  \begin{enumerate}
  \item for any closed local sections $\alpha$, $\beta$ of
    $\mathbb{F}$, the differential $d\left(\pi(\alpha, \beta)\right)$
    is a local section of $\mathbb{F}$;
  \item (Jacobi) for any closed local sections $\alpha$, $\beta$,
    $\gamma$ of $\mathbb{F}$,
    \begin{equation}\label{Jacobi_Poisson}
      \pi\left(\alpha, d\left(\pi(\beta, \gamma)\right)\right) + \pi\left(\beta, d\left(\pi(\gamma, \alpha)\right)\right) + \pi\left(\gamma, d\left(\pi(\alpha, \beta)\right) \right)= 0.
    \end{equation}
  \end{enumerate}
  The triple $(M, \mathbb{F}, \pi)$ will be called a \textbf{Banach
    sub-Poisson manifold}.
\end{definition}

\begin{remark}
  Given a Poisson tensor on a Banach manifold $M$, one can define a
  Poisson bracket on the space of locally defined functions with
  differentials in $\mathbb{F}$ by
  \begin{equation*}
    \{ f, g\} = \pi(df, dg).
  \end{equation*}
  Condition~1 in Definition~\ref{Poisson_tensor} ensures that the
  bracket of two such functions is again a function of the same type,
  and condition 2 is equivalent to the usual Jacobi
  identity. Consequently, the space of smooth functions on $M$ with
  differentials in $\mathbb{F}$ forms a Poisson algebra. Note that the
  existence of Hamiltonian vector fields is not generally assumed.
\end{remark}

\begin{remark}
  The notion of Banach sub-Poisson manifold is adapted to the
  infinite-dimensional context where~:
  \begin{enumerate}
  \item the tangent space of a Banach manifold may be in duality with
    many different Banach spaces. All these Banach spaces can be
    identified with subspaces of the cotangent space~;
  \item a Banach manifold $M$ may not have partition of unity or bump
    functions, hence it may not be possible to extend locally defined
    objects to global ones. This explains why we consider local
    sections instead of smooth functions on $M$ in order to define a
    Poisson structure on $M$.
  \end{enumerate}
\end{remark}

\begin{definition}\label{poisson_map_def}
  Let $(M_1, \mathbb{F}_1, \pi_1)$ and $(M_2, \mathbb{F}_2, \pi_2)$ be
  Banach sub-Poisson manifolds and $F:M_1\rightarrow M_2$ a smooth
  map. One says that $F$ is a \textbf{Poisson map} at $p\in M_1$ if
  \begin{enumerate}
  \item the tangent map $T_pF:T_pM_1 \rightarrow T_{F(p)}M_2$
    satisfies $T_pF^*(\mathbb{F}_2)_{F(p)}\subset (\mathbb{F}_1)_p$
    and $T_pF^*: (\mathbb{F}_2)_{F_2(p)} \rightarrow (\mathbb{F}_1)_p$ is
    continuous~;
  \item
    $(\pi_1)_p\left(T_pF^*(\alpha), T_pF^*(\beta)\right) =
    (\pi_2)_{F(p)}\left(\alpha, \beta\right)$ for any
    $\alpha, \beta \in (\mathbb{F}_2)_{F(p)}$.
  \end{enumerate}
  One says that $F$ is a Poisson map if it is a Poisson map at any
  $p\in M_1$.
\end{definition}

\begin{proposition}
  Let $(M_1, \mathbb{F}_1, \pi_1)$ and $(M_2, \mathbb{F}_2, \pi_2)$ be
  Banach sub-Poisson manifolds. Then the product $M_1\times M_2$
  carries a natural Banach sub-Poisson manifold structure
  $\left(M_1\times M_2, \mathbb{F}, \pi \right)$ where
  \begin{enumerate}
  \item $M_1\times M_2$ carries the product Banach manifold structure,
    in particular $T(M_1\times M_2)\simeq TM_1\oplus TM_2$ and
    $T^*(M_1\times M_2)\simeq T^*M_1\oplus T^*M_2$,
  \item $\mathbb{F}$ is the subbundle of $T^*M_1\oplus T^*M_2$ defined
    as
    \begin{equation*}
      \mathbb{F}_{(p, q)} = (\mathbb{F}_1)_p\oplus(\mathbb{F}_2)_q,
    \end{equation*}
  \item $\pi$ is defined on $\mathbb{F}$ by
    \begin{equation*}
      \pi(\alpha_1+\alpha_2, \beta_1 + \beta_2) = \pi_1(\alpha_1, \beta_1) +
      \pi_2(\alpha_2, \beta_2), \quad \alpha_1, \beta_1 \in \mathbb{F}_1,
      \alpha_2, \beta_2 \in \mathbb{F}_2.
    \end{equation*}
  \end{enumerate}
\end{proposition}

\begin{definition}\label{BPLG}
  A \textbf{Banach Poisson--Lie group} is a Banach Lie group $G$
  equipped with a Banach sub-Poisson manifold structure
  $(G, \mathbb{F}, \pi)$ such that the group multiplication
  $m: G \times G \rightarrow G$ is a Poisson map, where $G \times G$
  is endowed with the product sub-Poisson structure.
\end{definition}

\begin{remark}\label{remdual}
  Let $(G, \mathbb{F}, \pi)$ be a Banach Poisson--Lie group with
  Banach Lie algebra $\mathfrak{g}$ and unit element $e\in G$.
  According to Proposition~5.6 in \cite{tumpach-bruhat}, the
  compatibility condition between the Poisson tensor $\pi$ and the
  multiplication in $G$ implies that $G$ acts continuously on the
  fiber $\mathbb{F}_e \subset \mathfrak{g}^*$ by coadjoint action. By
  derivation, it follows that $\mathfrak g$ acts also continuously on
  $\mathbb{F}_e$ by coadjoint action.
\end{remark}

\section{Some subspaces of $\mathfrak{u}^*(\mathcal{H})$ in duality
  with $\mathfrak{u}(\mathcal{H})$}

In order to define a Banach Poisson--Lie group structure on the Banach
Lie group $\operatorname{U}(\mathcal{H})$, we are looking for
subspaces of the dual space $\mathfrak{u}^*(\mathcal{H})$ in duality
with $\mathfrak{u}(\mathcal{H})$, on which
$\operatorname{U}(\mathcal{H})$ acts continuously by coadjoint action
(see Remark~\ref{remdual}).

Endow the Hilbert space $\mathcal{H}$ with a Hilbert basis
$\{|n\rangle\}_{n\in\mathbb{Z}}$. We will
consider the following Banach Lie algebra of upper triangular
trace-class operators~:
\begin{equation*}
  \mathfrak{b}^+_1(\mathcal{H}) := \{\alpha\in
  \operatorname{L}_1(\mathcal{H}): \alpha |n\rangle \in~
  \textrm{span}\{|m\rangle, m\geq n\}~\textrm{and}~\langle
  n|\alpha|n\rangle\in\mathbb{R}, \textrm{for}~ n\in\mathbb{Z}\}.
\end{equation*}
This section is organized as follows. In subsection~\ref{bduality}, we
show that there is a duality pairing between
$\mathfrak{b}^+_1(\mathcal{H})$ and $\mathfrak{u}(\mathcal{H})$ and
prove that $\mathfrak{u}(\mathcal{H})$ does not act continuously on
$\mathfrak{b}^+_1(\mathcal{H})$ by coadjoint action, hence
$\mathfrak{b}^+_1(\mathcal{H})$ cannot be used to define a
Poisson--Lie group structure on $\operatorname{U}(\mathcal{H})$. In
subsection~\ref{stable}, we construct a subspace of
$\mathfrak{u}^*(\mathcal{H})$ into which
$\mathfrak{b}^+_1(\mathcal{H})$ injects as a dense subspace and on
which $\mathfrak{u}(\mathcal{H})$ acts continuously.

\subsection{Duality pairing between $\mathfrak{u}(\mathcal{H})$ and
  $\mathfrak{b}^+_1(\mathcal{H})$}\label{bduality}

\begin{proposition}\label{dualityp}
  The continuous bilinear map between $\mathfrak{u}(\mathcal{H})$ and
  $\mathfrak{b}_{1}^+(\mathcal{H})$ given by the imaginary part of the
  trace
  \begin{equation*}
    \begin{aligned}
      \mathfrak{u}(\mathcal{H}) \times \mathfrak{b}_{1}^+(\mathcal{H})& \rightarrow && \mathbb{R}\\
      (A, B) & \mapsto && \Im\Tr AB
    \end{aligned}
  \end{equation*}
  is non-degenerate, hence it defines a duality pairing between
  $\mathfrak{u}(\mathcal{H})$ and $\mathfrak{b}_{1}^+(\mathcal{H})$.
\end{proposition}

\begin{proof}
  It follows by direct calculation using e.g. operators
  $E_{nm} := |n\rangle\langle m |$.
\end{proof}

For finite-dimensional $\mathcal{H}$, this Proposition implies that
$\mathfrak{b}_{1}^+(\mathcal{H})$ can be identified with the dual of
$\mathfrak{u}(\mathcal{H})$. %
In the rest of the paper we will assume that $\mathcal{H}$ is
infinite-dimensional. In this case $\mathfrak{b}_{1}^+(\mathcal{H})$
can be identified with a proper subspace of
$\mathfrak{u}^*(\mathcal{H})$ using the duality pairing defined in
Proposition~\ref{dualityp}.
\begin{theorem}\label{b1}
  The coadjoint action of $\mathfrak{u}(\mathcal{H})$ is unbounded on
  the image of $\mathfrak{b}_{1}^+(\mathcal{H})$ in
  $ \mathfrak{u}^*(\mathcal{H})$.
\end{theorem}

\begin{proof}
  Denote by $T_+$ (resp. $T_{++}$) the linear transformation
  truncating an operator to the upper triangular part (resp. strictly
  upper triangular part) with respect to the Hilbert basis
  $\{|n\rangle\}_{n\in\mathbb{Z}}$~:
  \begin{align}\label{T++}
    \langle m | T_{++}(A) n\rangle &:= \begin{cases} \langle m | A n\rangle & \text{ if } m> n\\ 0 & \text{ if } m\leq n\end{cases}\\
    \label{T+}
    \langle m | T_{+}(A) n\rangle &:= \begin{cases} \langle m | A n\rangle & \textrm{ if } m\geq n\\ 0 & \text{ if } m<n\end{cases}
  \end{align}

  Recall that $T_+$ is unbounded on $L_\infty(\mathcal{H})$, as well
  as on $L_1(\mathcal{H})$ (see \cite{kwapien70}, \cite{gohberg70},
  \cite{davidson}), but they are bounded on the space of
  Hilbert--Schmidt operators $L_2(\mathcal H)$ since they are just
  orthogonal projections.

  Let us denote by $T_0$ the diagonal truncation defined by
  $T_0 = T_+ - T_{++}$, which is bounded on $L_\infty(\mathcal{H})$
  and $L_1(\mathcal{H})$.

  Let us consider the coadjoint action of $\mathfrak{u}(\mathcal{H})$
  on the image of $\mathfrak{b}_{1}^+(\mathcal{H})$ in
  $\mathfrak{u}^*(\mathcal{H})$.

  An element $B\in \mathfrak{b}_{1}^+(\mathcal{H})$ defines a
  functional $C\mapsto \Im\Tr CB$, $C \in \mathfrak{u}(\mathcal{H})$,
  on which $A\in \mathfrak{u}(\mathcal{H})$ acts by coadjoint action
  as:
  \begin{equation*}
    C\mapsto \Im\Tr [A, C]B = -\Im \Tr C[A, B]. %
  \end{equation*}
  Since for any $C \in \mathfrak{u}(\mathcal{H})$ and
  $D\in L_1(\mathcal{H})$,
  \begin{equation*}
    \Im \Tr C D = \Im \Tr C \big((T_{++} + \tfrac{1}{2}T_{0})(D + D^*
    )\big),
  \end{equation*}
  we have
  \begin{equation}\label{b+coad}
    \ad^*_A B = - (T_{++} + \tfrac{1}{2} T_{0})([A, B] + [A, B]^* ).%
  \end{equation}

  We show that this coadjoint action is unbounded on
  $\mathfrak{b}_{1}^+(\mathcal{H})$. To this end let us decompose the
  Hilbert space $\mathcal{H}$ into the sum of two orthogonal
  infinite-dimensional closed subspaces~:
  \begin{equation*}
    \mathcal{H} = \mathcal{H}_+\oplus\mathcal{H}_-,
  \end{equation*}
  where $\mathcal{H}_+$ is the Hilbert space generated by
  $\{|n\rangle\}_{n\in\mathbb{N}\cup\{0\}}$ and $\mathcal{H}_-$ is the
  Hilbert space generated by
  $\{|n\rangle\}_{-n\in\mathbb{N}\setminus\{0\}}$. Let us define a
  unitary operator $u:\mathcal{H}_-\rightarrow \mathcal{H}_+$ by
  $u |-n\rangle = |n-1\rangle$, $n\in\mathbb{N}$.

  From unboundedness of $T_+$ it follows that there exists a sequence
  of Hermitian trace class operators $K_n$ on $\mathcal{H}_-$ such
  that
  \begin{enumerate}
  \item $\|K_{n}\|_1\leq 1$,
  \item $\lim\limits_{n\rightarrow+\infty} \|T_+(K_n)\|_1 = +\infty$.%
  \end{enumerate}

  The bounded linear operators whose expressions with respect to the
  decomposition $\mathcal{H} = \mathcal{H}_+\oplus\mathcal{H}_-$ read
  \begin{equation*}
    B_n := \begin{pmatrix} 0 & uK_n\\0 & 0 \end{pmatrix}
  \end{equation*}
  belong to $\mathfrak{b}_{1}^+(\mathcal{H})$. The skew-hermitian
  operator $A$ whose expression with respect to the decomposition
  $\mathcal{H} = \mathcal{H}_+\oplus\mathcal{H}_-$ reads
  \begin{equation*}
    A := \begin{pmatrix}0 & u\\-u^* & 0 \end{pmatrix}
  \end{equation*}
  is bounded. Moreover
  \begin{equation*} [A, B_n] = \left[\begin{pmatrix}0 & u\\-u^* &
        0 \end{pmatrix}, \begin{pmatrix}0 & uK_n\\0 &
        0 \end{pmatrix}\right]=
    \begin{pmatrix} uK_nu^* & 0\\0 & -K_n \end{pmatrix}.
  \end{equation*}
  Since $K_n$ is Hermitian, from \eqref{b+coad} we get
  \begin{equation*}
    \ad^*_A B_n = - (T_{++} + \tfrac{1}{2} T_{0})([A, B_n]).
  \end{equation*}
  It follows that $\|\ad^*_A B_n\|_1\rightarrow \infty$ as
  $n \rightarrow \infty$ whereas $\|B_n\|_1 = \|K_{n}\|_1\leq 1$.
\end{proof}

We conclude from Remark~\ref{remdual} the following corollary.
\begin{corollary}
  There is no Banach Poisson--Lie group structure
  $\left(\operatorname{U}(\mathcal{H}), \mathbb{F}, \pi\right)$ on
  $\operatorname{U}(\mathcal{H})$ with bundle $\mathbb{F}$ such that
  $\mathbb{F}_e = \mathfrak{b}^+_1(\mathcal{H})\subset
  \mathfrak{u}(\mathcal{H})^*$.
\end{corollary}

\subsection{A subspace of $\mathfrak{u}^*(\mathcal{H})$ on which
  $\mathfrak{u}(\mathcal{H})$ acts continuously by coajoint
  action}\label{stable}

Consider the continuous linear map
$\Phi: L_{1}(\mathcal{H}) \rightarrow \mathfrak{u}^*(\mathcal{H})$
which maps an operator $a\in L_{1}(\mathcal{H})$ to the linear
functional on $\mathfrak{u}(\mathcal{H})$ given by
\begin{equation*}
  b \mapsto \Im \Tr a b
\end{equation*}
where $\Im \Tr a b $ is the imaginary part of the trace of
$a b \in L_{1}(\mathcal{H})$.
\begin{proposition}
  The kernel of $\Phi$ equals $\mathfrak{u}_{1}(\mathcal{H})$,
  therefore $L_{1}(\mathcal{H})/\mathfrak{u}_{1}(\mathcal{H})$ injects
  into the dual space $\mathfrak{u}(\mathcal{H})^*$ and can be
  identified with the predual space $\mathfrak{u}(\mathcal{H})_*$. It
  is closed and preserved by the coadjoint action of
  $\operatorname{U}(\mathcal{H})$. Moreover functionals given by
  elements of $\mathfrak{b}_{1}^+(\mathcal{H})$ form a proper dense
  subspace.
\end{proposition}

\begin{proof}$\;$
  \begin{itemize}
  \item It is clear that $\mathfrak{u}_{1}(\mathcal{H})$ is contained
    in the kernel of $\Phi$ since the product of two skew-hermitian
    operators has a real trace. Let $a\in L_{1}(\mathcal{H})$ be such
    that $\Im \Tr a b = 0$ for any $b\in \mathfrak{u}(\mathcal{H})$.

    One has for any $b \in \mathfrak{u}(\mathcal{H})$ and
    $a\in L_1(\mathcal{H})$
    \begin{equation*}
      \Im \Tr a b = \tfrac{1}{2i}\left(\Tr a b - \overline{\Tr a b}\right) =
      \tfrac{1}{2i}\left(\Tr a b + \Tr b a^* \right) = %
      \tfrac{1}{2i}\left(\Tr (a + a^*) b \right).
    \end{equation*}
    Consequently,
    \begin{equation*}
      a\in \operatorname{ker}(\Phi) \Rightarrow \Tr (a + a^*) b = 0 \quad
      \forall b\in \mathfrak{u}(\mathcal{H}).
    \end{equation*}
    By linearity of the trace,
    $\Tr (a + a^*) b = 0 \Rightarrow \Tr (a + a^*) ib = 0$. It follows
    that if $a\in \operatorname{ker}(\Phi)$,
    $\Tr (a + a^*) \tilde b = 0$ for any
    $\tilde b \in L_{\infty}(\mathcal{H})$. Since the dual of
    $L_1(\mathcal{H})$ viewed as complex Banach space can be
    identified with $ L_{\infty}(\mathcal{H})$ using the trace, one
    has
    \begin{equation*}
      a\in \operatorname{ker}(\Phi)\Rightarrow a + a^* = 0\in L_1(\mathcal{H}).
    \end{equation*}
    Hence the kernel of $\Phi$ equals $\mathfrak{u}_{1}(\mathcal{H})$.
  \item From the previous point, we have an injection
    \begin{equation}\label{injection}
      L_1(\mathcal{H})/\mathfrak{u}_{1}(\mathcal{H}) \hookrightarrow \mathfrak{u}(\mathcal{H})^*.
    \end{equation}

    From the Banach decomposition
    \begin{equation}\label{L1dec}
      L_{1}(\mathcal{H}) = \mathfrak{u}_1(\mathcal{H}) \oplus i\mathfrak{u}_1(\mathcal{H})
    \end{equation}
    given by
    $a\mapsto \left(\tfrac{1}{2}(a-a^*)~; \tfrac{1}{2}(a+a^*)\right)$,
    it follows that
    \begin{equation}\label{first}
      L_1(\mathcal{H})/\mathfrak{u}_{1}(\mathcal{H})\simeq i\mathfrak{u}_{1}(\mathcal{H}).
    \end{equation}
    It is known that
    \begin{equation*}
      (i\mathfrak{u}_{1}(\mathcal{H}))^*\simeq \mathfrak{u}(\mathcal{H}),
    \end{equation*}
    see e.g. \cite[Example 7.10]{OR}.

    Thus the injection~\eqref{injection} is in fact the natural
    injection of the Banach space
    $L_1(\mathcal{H})/\mathfrak{u}_{1}(\mathcal{H})$ into its bidual
    \begin{equation*}
      \left(L_1(\mathcal{H})/\mathfrak{u}_{1}(\mathcal{H})\right)^{**}\simeq
      \mathfrak{u}(\mathcal{H})^*.
    \end{equation*}
    Its image is therefore a closed subspace of
    $\mathfrak{u}(\mathcal{H})^*.$

  \item Let us show that the range of $\Phi$ is preserved by the
    coadjoint action of $\operatorname{U}(\mathcal{H})$. For
    $a\in L_1(\mathcal{H})$ and $b\in \mathfrak{u}(\mathcal{H})$, one
    has
    \begin{align*}
      \Ad^*_g(\Phi(a))(b) & = \Phi(a)(\Ad_g b) = \Im \Tr a gbg^{-1}\\ & = \Im \Tr g^{-1} a g b = \Phi(g^{-1} a g)(b),
    \end{align*}
    where $g^{-1} a g$ belongs to $L_1(\mathcal{H})$. Note that for
    $a\in \mathfrak{u}_{1}(\mathcal{H})$, $g^{-1} a g$ belongs to
    $\mathfrak{u}_{1}(\mathcal{H})$ for any
    $g\in \operatorname{U}(\mathcal{H})$. Hence the coadjoint action
    of $g\in\operatorname{U}(\mathcal{H})$ on
    $L_1(\mathcal{H})/\mathfrak{u}_{1}(\mathcal{H})$ reads
    \begin{equation}\label{coad_l/u}
      \Ad^*_g[a]_{\mathfrak{u}_1} = [g^{-1} a g]_{\mathfrak{u}_1},
    \end{equation}
    where $[a]_{\mathfrak{u}_1}$ denotes the class of
    $a\in L_1(\mathcal{H})$ modulo $\mathfrak{u}_{1}(\mathcal{H})$.

  \item Since
    $\mathfrak{b}_1^+(\mathcal{H}) \cap \mathfrak{u}_{1}(\mathcal{H})
    = \{0\}$, we have
    $\mathfrak{b}_1^+(\mathcal{H}) \hookrightarrow
    L_1(\mathcal{H})/\mathfrak{u}_{1}(\mathcal{H}) \simeq
    i\mathfrak{u}_1(\mathcal{H})$. Under this identification an
    element $b\in \mathfrak{b}_1^+(\mathcal{H}) $ is sent to
    $\tfrac{1}{2}(b+b^*)$. Hence $\mathfrak{b}_1^+(\mathcal{H})$
    corresponds to those elements in $i\mathfrak{u}_1(\mathcal{H})$
    that have a triangular truncation in $L_1(\mathcal{H})$.

    Any functional
    $a \in \mathfrak{u}(H)\simeq
    (L_1(\mathcal{H})/\mathfrak{u}_{1}(\mathcal{H}))^* $ vanishing on
    $\mathfrak{b}_1^+(\mathcal{H})$ by Proposition~\ref{dualityp} is
    zero. Hence $\mathfrak{b}_1^+(\mathcal{H})$ is dense in
    $L_1(\mathcal{H})/\mathfrak{u}_{1}(\mathcal{H})$.
  \end{itemize}
\end{proof}

\section{The unitary group $U(\mathcal{H})$ as a Banach Poisson--Lie
  group}

In order to define a Banach Poisson--Lie group structure on
$U(\mathcal{H})$ we need to introduce the Lie algebra
$\mathfrak{u}_2(\mathcal{H})$ of Hilbert--Schmidt skew-hermitian
operators, as well as the Lie algebra $\mathfrak{b}^+_2(\mathcal{H})$
of Hilbert--Schmidt upper triangular operators with real coefficients
on the diagonal. Note that we have a direct sum decomposition of the
space $L_2(\mathcal{H})$ of Hilbert--Schmidt operators into
\begin{equation*}
  L_2(\mathcal{H}) =
  \mathfrak{u}_2(\mathcal{H})\oplus\mathfrak{b}^+_2(\mathcal{H}).
\end{equation*}
The corresponding projections $p_{\mathfrak{u}_2}$ and
$p_{\mathfrak{b}^+_2}$ from $L_2(\mathcal{H})$ onto
$\mathfrak{u}_2(\mathcal{H})$ and $\mathfrak{b}^+_2(\mathcal{H})$
respectively are continuous.

Since $\mathfrak{u}_2(\mathcal{H})$ is invariant by conjugation by a
unitary operator, but $\mathfrak{b}^+_2(\mathcal{H})$ is not, one has
for $x \in L_1(\mathcal{H})$ and $g \in U(\mathcal{H})$,
\begin{equation}\label{b}
  p_{\mathfrak{b}_2^+}(g^{-1} x\, g)
  = p_{\mathfrak{b}_2^+}(g^{-1} p_{\mathfrak{b}_2^+}(x)\, g),
\end{equation}
and
\begin{equation}\label{u}
  p_{\mathfrak{u}_2}(g^{-1} x\, g) = g^{-1} p_{\mathfrak{u}_2}(x)\, g + p_{\mathfrak{u}_2}(g^{-1} p_{\mathfrak{b}_2^+}(x)\, g).
\end{equation}

\begin{theorem}\label{UresPoisson}
  Consider the Banach Lie group $\operatorname{U}(\mathcal{H})$ and
  define
  \begin{itemize}
  \item the precotangent bundle
    $\mathbb{F}= T_*\operatorname{U}(\mathcal{H}) \subset
    T^*\operatorname{U}(\mathcal{H})$ by right translations
    \begin{equation*}
      \mathbb{F}_g = R_{g^{-1}}^* \big(\operatorname{L}_{1}(\mathcal{H})/\mathfrak{u}_{1}(\mathcal{H})\big) = R_{g^{-1}}^* \mathfrak{u}(\mathcal{H})_*,
    \end{equation*}
  \item the map
    $\Pi_r:\operatorname{U}(\mathcal{H})\rightarrow
    \Lambda^2\mathbb{F}_e^*$ by
    \begin{equation}\label{Pig}
      \Pi_r(g)([x_1]_{\mathfrak{u}_{1}}, [x_2]_{\mathfrak{u}_{1}}) = \Im\Tr\big(g^{-1}\,p_{\mathfrak{b}^+_2}(x_1)\,g\; p_{\mathfrak{u}_2}(g^{-1}\,p_{\mathfrak{b}^+_2}(x_2)\,g)\big),
    \end{equation}
  \item the tensor $\pi\in\Lambda^2\mathbb{F}^*$ by
    $\pi(g) = R^{**}_g \Pi_r(g)$.
  \end{itemize}
  Then $\left(\operatorname{U}(\mathcal{H}), \mathbb{F}, \pi\right)$
  is a Banach Poisson--Lie group. On the space of smooth functions
  with differentials in $\mathbb{F}$, the Poisson bracket reads:
  \begin{equation*}
    \{f, h\}(g) = %
    \Pi_r(g)(R^{*}_g df_g, R^{*}_g dh_g) = \Pi_r(g)(df_g\circ R_g,
    dh_g\circ R_g).
  \end{equation*}
\end{theorem}

\begin{proof}
  We need to check that $\pi$ is compatible with the group
  multiplication and satisfies the Jacobi identity.
  \begin{enumerate}
  \item Using Proposition~5.7 in \cite{tumpach-bruhat}, the
    compatibility with the group multiplication is equivalent to the
    fact that $\Pi_r$ satisfies the following cocycle condition:
    \begin{equation}\label{cocy}
      \Pi_r(gu) = (\Ad_g^*)^{*}\Pi_r(u) + \Pi_r(g),
    \end{equation}
    where $(\Ad_g^*)^*$ denotes the natural action of
    $g \in U(\mathcal{H})$ on $\Lambda^2\mathbb{F}^*_e$ given
    explicitly by
    \begin{equation*}
      (\Ad_g^*)^*\Pi_r(u)\left([x_1]_{\mathfrak{u}_1},
        [x_2]_{\mathfrak{u}_1}\right) =
      \Pi_r(u)\left(\Ad^*_g[x_1]_{\mathfrak{u}_1},
        \Ad^*_g[x_2]_{\mathfrak{u}_1}\right).
    \end{equation*}
    In order to check that condition, we use \eqref{coad_l/u} and
    \eqref{Pig}:
    \begin{multline*}
      (\Ad_g^*)^*\Pi_r(u)\left([x_1]_{\mathfrak{u}_1}, [x_2]_{\mathfrak{u}_1}\right)
      = \Pi_r(u) \left([g^{-1} \, x_1\,g]_{\mathfrak{u}_1}, [g^{-1}x_2 \, g] _{\mathfrak{u}_1} \right)
      \\=\Im \Tr \big(u^{-1} p_{\mathfrak{b}_2^+}(g^{-1} x_1 \,g)\,u\; p_{\mathfrak{u}_2}(u^{-1} p_{\mathfrak{b}_2^+}(g^{-1} x_2\, g) u)\big).
    \end{multline*}
    Using \eqref{b}, this expression can be further written as
    \begin{equation*}
      \Im \Tr \big(u^{-1} p_{\mathfrak{b}_2^+}(g^{-1} p_{\mathfrak{b}_2^+}(x_1) \,g) \, u \; p_{\mathfrak{u}_2}(u^{-1} p_{\mathfrak{b}_2^+}(g^{-1} p_{\mathfrak{b}_2^+}(x_2) \,  g) u)\big).
    \end{equation*}
    Using the decomposition
    \begin{equation}\label{x1}
      p_{\mathfrak{b}_2^+}(g^{-1} p_{\mathfrak{b}_2^+}(x_1) \,g) = g^{-1} p_{\mathfrak{b}_2^+}(x_1) \,g - p_{\mathfrak{u}_2}(g^{-1} p_{\mathfrak{b}_2^+}(x_1) \,g)
    \end{equation}
    and the fact that $\mathfrak{u}_2(\mathcal{H})$ is isotropic for
    the imaginary part of the trace, one has
    \begin{multline*}
      (\Ad_g^*)^*\Pi_r(u)\left([x_1]_{\mathfrak{u}_1}, [x_2]_{\mathfrak{u}_1}\right)
      \\ = \Im \Tr \big(u^{-1} g^{-1} p_{\mathfrak{b}_2^+}(x_1) \,g \,u \; p_{\mathfrak{u}_2}(u^{-1} p_{\mathfrak{b}_2^+}(g^{-1} p_{\mathfrak{b}_2^+}(x_2)\, g) u)\big).
    \end{multline*}
    Using equation~\eqref{x1} for $x_2$, one finally gets
    \begin{multline*}
      (\Ad_g^*)^*\Pi_r(u)\left([x_1]_{\mathfrak{u}_1}, [x_2]_{\mathfrak{u}_1}\right)
      \\ = \Im \Tr \big(u^{-1} g^{-1} p_{\mathfrak{b}_2^+}(x_1) \,g \,u \; p_{\mathfrak{u}_2}(u^{-1} g^{-1} p_{\mathfrak{b}_2^+}(x_2)\, g u)\big)
      \\- \Im \Tr \big(u^{-1} g^{-1} p_{\mathfrak{b}_2^+}(x_1) \,g \,u \; p_{\mathfrak{u}_2}(u^{-1} p_{\mathfrak{u}_2}(g^{-1} p_{\mathfrak{b}_2^+}(x_2)\, g) u)\big),
    \end{multline*}
    which, after simplification by $u$ in the last term, gives
    identity~\eqref{cocy}.
  \item Note that by construction the sharp map
    $\sharp: \mathbb{F} \rightarrow \mathbb{F}^*$,
    $\alpha \mapsto \pi(\alpha, \cdot)$ takes values in the tangent
    space of $U(\mathcal{H})$. Therefore, in order to check that $\pi$
    satisfies the Jacobi identity, we can use formula~(5.5) from
    Lemma~5.8 in \cite{tumpach-bruhat}. We will show that
    \begin{gather} \label{pigr} T_g\Pi_r(R_g
      \iota_{[x_3]_{\mathfrak{u}_1}}\Pi_r(g))([x_1]_{\mathfrak{u}_1},
      [x_2]_{\mathfrak{u}_1})
      = -\Im \Tr p_{\mathfrak{u}_2}(C) \; [p_{\mathfrak{b}_2^+}(A), p_{\mathfrak{b}_2^+}(B)]\\
      \label{lr1}
      \langle x_1, [\iota_{[x_3]}\Pi_r(g), \iota_{[x_2]}\Pi_r(g)]\rangle %
      = -\Im \Tr p_{\mathfrak{u}_2}(C) [p_{\mathfrak{b}_2^+}(A), p_{\mathfrak{u}_2}(B)],
    \end{gather}
    where $A = g^{-1} p_{\mathfrak{b}_2^+}(x_1) g$,
    $B = g^{-1}\,p_{\mathfrak{b}_2^+}(x_2)\, g$ and
    $C = g^{-1}\,p_{\mathfrak{b}_2^+}(x_3)\, g)$. Jacobi identity will
    then follow by adding the terms obtained by circular permutations
    of equations~\eqref{pigr} and \eqref{lr1}, and remarking that
    \begin{gather*}
      -\Im \Tr p_{\mathfrak{u}_2}(C) \; [p_{\mathfrak{b}_2^+}(A), p_{\mathfrak{b}_2^+}(B)] -\Im \Tr p_{\mathfrak{u}_2}(C) [p_{\mathfrak{b}_2^+}(A), p_{\mathfrak{u}_2}(B)]\\
      -\Im \Tr p_{\mathfrak{u}_2}(A) \; [p_{\mathfrak{b}_2^+}(B), p_{\mathfrak{b}_2^+}(C)] -\Im \Tr p_{\mathfrak{u}_2}(A) [p_{\mathfrak{b}_2^+}(B), p_{\mathfrak{u}_2}(C)]\\
      -\Im \Tr p_{\mathfrak{u}_2}(B) \; [p_{\mathfrak{b}_2^+}(C), p_{\mathfrak{b}_2^+}(A)] -\Im \Tr p_{\mathfrak{u}_2}(B) [p_{\mathfrak{b}_2^+}(C), p_{\mathfrak{u}_2}(A)]\\
      = -\Im \Tr p_{\mathfrak{u}_2}(C) \; [A, B] - \Im \Tr p_{\mathfrak{b}_2^+}(C)\; [A, B]
      = -\Im \Tr C\; [A, B] \\ = -\Im \Tr g^{-1}\,p_{\mathfrak{b}_2^+}(x_3)\, g\; [g^{-1} p_{\mathfrak{b}_2^+}(x_1) g, g^{-1}\,p_{\mathfrak{b}_2^+}(x_2)\, g]
      \\= -\Im\Tr p_{\mathfrak{b}_2^+}(x_3) [p_{\mathfrak{b}_2^+}(x_1), p_{\mathfrak{b}_2^+}(x_2)] = 0,
    \end{gather*}
    where the last equality follows from the fact that
    $\mathfrak{b}_2^+(\mathcal{H})$ is an isotropic subalgebra. In
    order to prove equations~\eqref{pigr} and \eqref{lr1}, one needs
    three ingredients:
    \begin{enumerate}
    \item The differentiation of the cocycle identity~\eqref{cocy}
      with respect to $u$ leads to the following identity
      \begin{equation}\label{a}
        T_g\Pi_r(L_{g}X)([x_1]_{\mathfrak{u}_1}, [x_2]_{\mathfrak{u}_1}) = T_e\Pi_r(X)(\Ad^*_g[x_1]_{\mathfrak{u}_1}, \Ad^*_g[x_2]_{\mathfrak{u}_1}),
      \end{equation}
      where $X\in\mathfrak{u}(\mathcal{H})$ and $g\in U(\mathcal{H})$.
    \item The differentiation of equation~\eqref{Pig} with respect to
      $g\in U(\mathcal{H})$ gives
      \begin{equation}\label{b2}
        T_e\Pi_r(Y)([x_1]_{\mathfrak{u}_1}, [x_2]_{\mathfrak{u}_1}) = \Im\Tr Y[p_{\mathfrak{b}_2^+}(x_1), p_{\mathfrak{b}_2^+}(x_2)].
      \end{equation}
    \item By equation~\eqref{Pig}, the interior product of $\Pi_r(g)$
      with $[x_3]_{\mathfrak{u}_1}$ is %
      \begin{equation}\label{c}
        \iota_{[x_3]_{\mathfrak{u}_1}}\Pi_r(g) =- g\,p_{\mathfrak{u}_2}(g^{-1}\,p_{\mathfrak{b}_2^+}(x_3)\, g)\,g^{-1} \in \mathfrak{u}(\mathcal{H}).
      \end{equation}
    \end{enumerate}
    From equation~\eqref{a}, it follows that
    \begin{align*}
      T_g\Pi_r(R_{g}X)([x_1]_{\mathfrak{u}_1}, [x_2]_{\mathfrak{u}_1}) &= T_g\Pi_r(L_{g}\Ad_{g^{-1}}(X))([x_1]_{\mathfrak{u}_1}, [x_2]_{\mathfrak{u}_1}) \\ & =T_e\Pi_r(\Ad_{g^{-1}}(X))(\Ad^*_g[x_1]_{\mathfrak{u}_1}, \Ad^*_g[x_2]_{\mathfrak{u}_1})\\& =
      T_e\Pi_r(\Ad_{g^{-1}}(X))([g^{-1}\,x_1\,g]_{\mathfrak{u}_1}, [g^{-1}\,x_2\,g]_{\mathfrak{u}_1})
    \end{align*}
    Using equation~\eqref{b2}, this simplifies to
    \begin{equation}\label{Rgpi}
      T_g\Pi_r(R_gX)([x_1]_{\mathfrak{u}_1}, [x_2]_{\mathfrak{u}_1}) = \Im \Tr g^{-1}\,X\, g[p_{\mathfrak{b}_2^+}(g^{-1}\,x_1\,g), p_{\mathfrak{b}_2^+}(g^{-1}\,x_2\,g)].
    \end{equation}
    Hence for $X =\iota_{[x_3]_{\mathfrak{u}_1}}\Pi_r(g)$, using
    equation~\eqref{c} one gets
    \begin{multline}\label{Teq}
      T_g\Pi_r(R_g \iota_{[x_3]_{\mathfrak{u}_1}}\Pi_r(g))([x_1]_{\mathfrak{u}_1}, [x_2]_{\mathfrak{u}_1})\\ =
      -\Im \Tr \,p_{\mathfrak{u}_2}(g^{-1}\,p_{\mathfrak{b}_2^+}(x_3)\, g)[p_{\mathfrak{b}_2^+}(g^{-1}\,x_1\,g), p_{\mathfrak{b}_2^+}(g^{-1}\,x_2\,g)]
    \end{multline}
    By equation~\eqref{b} and by the isotropy of $\mathfrak{b}_2^+$,
    one gets
    \begin{multline}
      T_g\Pi_r(R_g \iota_{[x_3]_{\mathfrak{u}_1}}\Pi_r(g))([x_1]_{\mathfrak{u}_1}, [x_2]_{\mathfrak{u}_1}) \\
      = -\Im \Tr g^{-1}\,p_{\mathfrak{b}_2^+}(x_3)\, g \; [p_{\mathfrak{b}_2^+}(g^{-1}\,p_{\mathfrak{b}_2^+}(x_1)\,g), p_{\mathfrak{b}_2^+}(g^{-1}\,p_{\mathfrak{b}_2^+}(x_2)\,g)],
      \\
      = -\Im \Tr p_{\mathfrak{u}_2}(g^{-1}\,p_{\mathfrak{b}_2^+}(x_3)\, g) \; [p_{\mathfrak{b}_2^+}(g^{-1}\,p_{\mathfrak{b}_2^+}(x_1)\,g), p_{\mathfrak{b}_2^+}(g^{-1}\,p_{\mathfrak{b}_2^+}(x_2)\,g)],
    \end{multline}
    which is equation~\eqref{pigr}. On the other hand, by
    equation~\eqref{c} one gets
    \begin{multline*}
      [\iota_{[x_3]}\Pi_r(g), \iota_{[x_2]}\Pi_r(g)]\\ = [- g\,p_{\mathfrak{u}_2}(g^{-1}\,p_{\mathfrak{b}_2^+}(x_3)\, g)\,g^{-1} , - g\,p_{\mathfrak{u}_2}(g^{-1}\,p_{\mathfrak{b}_2^+}(x_2)\, g)\,g^{-1} ]
      \\ = g [\,p_{\mathfrak{u}_2}(g^{-1}\,p_{\mathfrak{b}_2^+}(x_3)\, g) , \,p_{\mathfrak{u}_2}(g^{-1}\,p_{\mathfrak{b}_2^+}(x_2)\, g)]\,g^{-1}.
    \end{multline*}
    Hence
    \begin{multline}\label{lr}
      \langle x_1, [\iota_{[x_3]}\Pi_r(g), \iota_{[x_2]}\Pi_r(g)]\rangle \\ = \Im \Tr x_1 g [\,p_{\mathfrak{u}_2}(g^{-1}\,p_{\mathfrak{b}_2^+}(x_3)\, g), \,p_{\mathfrak{u}_2}(g^{-1}\,p_{\mathfrak{b}_2^+}(x_2)\, g)]\,g^{-1} \\
      = \Im \Tr g^{-1} p_{\mathfrak{b}_2^+}(x_1) g [p_{\mathfrak{u}_2}(g^{-1}\,p_{\mathfrak{b}_2^+}(x_3)\, g), p_{\mathfrak{u}_2}(g^{-1}\,p_{\mathfrak{b}_2^+}(x_2)\, g)]\\
      = -\Im \Tr p_{\mathfrak{b}_2^+}(g^{-1} p_{\mathfrak{b}_2^+}(x_1) g) [p_{\mathfrak{u}_2}(g^{-1}\,p_{\mathfrak{b}_2^+}(x_2)\, g), p_{\mathfrak{u}_2}(g^{-1}\,p_{\mathfrak{b}_2^+}(x_3)\, g)].
    \end{multline}
    By the compatibility of the bracket of operators with the
    trace\linebreak ($\Tr A[ B, C] = \Tr C [A, B]$), this can be
    rewritten as
    \begin{multline}
      \langle x_1, [\iota_{[x_3]}\Pi_r(g), \iota_{[x_2]}\Pi_r(g)]\rangle \\
      = -\Im \Tr p_{\mathfrak{u}_2}(g^{-1}\,p_{\mathfrak{b}_2^+}(x_3)\, g) [p_{\mathfrak{b}_2^+}(g^{-1} p_{\mathfrak{b}_2^+}(x_1) g) ,p_{\mathfrak{u}_2}(g^{-1}\,p_{\mathfrak{b}_2^+}(x_2)\, g)],
    \end{multline}
    which proves equation~\eqref{lr1}.
  \end{enumerate}
\end{proof}

\begin{remark}
  The Lie bracket on
  $L^1(\mathcal{H})/\mathfrak{u}_1(\mathcal{H}) =
  \mathfrak{u}_*(\mathcal{H})$ that the Poisson--Lie group structure
  of $U(\mathcal{H})$ induces by Theorem~(5.11) in
  \cite{tumpach-bruhat} is given by
  \begin{equation}\label{bracketb}
    ([x_1]_{\mathfrak{u}_{1}}, [x_2]_{\mathfrak{u}_{1}}) \mapsto [p_{\mathfrak{b}_2^+}(x_1), p_{\mathfrak{b}_2^+}(x_2)],
  \end{equation}
  which is well defined on
  $\operatorname{L}_{1}(\mathcal{H})/\mathfrak{u}_{1}(\mathcal{H})$
  since
  $[p_{\mathfrak{b}_2^+}(x_1), p_{\mathfrak{b}_2^+}(x_2)]\in
  \operatorname{L}_1(\mathcal{H})$ for any
  $x_1, x_2\in \operatorname{L}_{1}(\mathcal{H})$. Note that this
  bracket is continuous and extends the natural bracket of
  $\mathfrak{b}_{1}^+(\mathcal{H})$. To our knowledge it is an open
  question whether this Banach Lie algebra structure integrates to a
  Banach Lie group.
\end{remark}

\end{document}